\documentclass[12pt]{article}
\usepackage[active]{srcltx}
\usepackage{graphicx}
\usepackage[centertags]{amsmath}
\usepackage{amsfonts}
\usepackage{amssymb}
\usepackage{amsthm}
\newcommand{\Title}[1]{{\Large \bf \begin{center} {#1}
  \end{center}}\vspace*{1mm}}
\newcommand{\Author}[1]{{\bf \begin{center} {#1}
  \end{center}}\vspace*{1mm}}
\newtheorem{thm}{Theorem} 

\newtheorem{defn}[thm]{Definition}

\newcommand{\prmb}[2]{\left\{ \begin{array}{ll} {C_1:}  & {#1} \\  {C_2:}  &  {#2} \end{array}   \right.}

\begin{document}
\Title{\Large Combinatorial Maps with Normalized Knot} \Author{ Dainis ZEPS \footnote{This research was supported in part by the Ministry of Education and Science
Republic of Latvia, Project 09.1247.}
\footnote{Author's address: Institute of Mathematics and Computer
Science, University of Latvia, 29 Rainis blvd., Riga, Latvia.
{dainis.zeps@lumii.lv}} }

\begin{abstract}
We consider combinatorial maps with fixed combinatorial knot numbered with augmenting numeration called normalized knot. We show that knot's normalization doesn't affect combinatorial map what concerns its generality. Knot's normalization leads to more concise numeration of corners in maps, e.g., odd or even corners allow easy to follow distinguished cycles in map caused by the fixation of the knot. \\
Knot's normalization may be applied to edge structuring knot too. If both are normalized then one is fully and other partially normalized mutually.
\end{abstract}

{\bf Keywords}: graphs on surface, zigzag walk, permutations, combinatorial maps,  combinatorial knots.

\section{Introduction}
Our approach of combinatorial maps is based on works \cite{he 1891,co 75,sta 78,sta 88,co 75} and developed in series of articles \cite{ze 94,ze 96,ze 97,ze 08}, online book \cite{ze 04} and PhD thesis \cite{ze 98}.

Let us use definitions introduced in works \cite{ze 94,ze 98,ze 04,ze 08}. Normalized combinatorial map is defined by arbitrary permutation $P$, i.e., $P$ is vertex rotation of the combinatorial map. Its inner edge rotation is involution $\pi = (1 2)(34) ... (2m-1)(2m)$ where $m$ is number of edges in the map. Face rotation $Q$ of the map we get by multiplication $P$  by $ \pi$: $$Q = P \times \pi.$$ Correspondingly, edge involution $\rho$ we get by multiplication $P$ by $Q^{-1}$: $$\rho = P \times Q^{-1},$$
or $$\rho = \pi^{P^{-1}}.$$

Combinatorial knot $\mu$ is defined as alternating application of $\pi$ and $\rho$ \cite{ze 98,ze 96}. Combinatorial knot may be written in the form \cite{ze 98} $$\mu=\prmb \pi \rho,$$
that should be read in the following way: by fixing knot $\mu$ corner set $C$ is partitioned into two parts, correspondingly $C_1$ and $C_2$, where, in order to get knot $\mu$, $\pi$ is applied from $C_1$ to $C_2$ and $\rho$ -- conversely.

\begin{figure}[h]
\begin{center}
\scalebox{.35}{\includegraphics{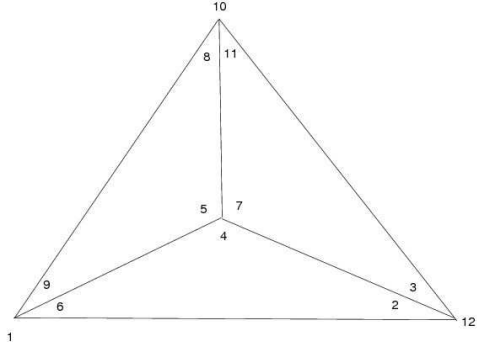}}
\caption{Example of combinatorial map with fixed knot. $P=(1 9 6)(2 3 \bar{2})(4 5 7)(8 \bar{0} \bar{1})$, $Q=1 \bar{0} \bar{2})(2 4 6)(3 \bar{1} 7)(5 8 9)$, edge involution is $(1 8)(5 \bar{1})(2 7)(6 \bar{2})(3 \bar{0})(4 9)$, knot is $\mu=(2 1 7 8)(4 3 \bar{0} 9)(6 5 \bar{1} \bar{2})$. Corners are partitioned into subsets $C_1=\{1,3,5,7,9,\bar{2}\}$ and $C_2=\{2,4, 6,8,\bar{0},\bar{1}\}$. Inner edge involution is partitioned into cut edges $\pi_1=(7 8)(9 \bar{0})(\bar{1} \bar{2})$ and cycle edges $\pi_2=(1 2)(3 4)(5 6)$. Fixed by  knot cycles in map $P$ are $\gamma_1=(1 9 5 7 3 \bar{2})$ and $\gamma_2=(2 4 6)(8 \bar{0} \bar{1})$. }
\label{fig1}
\end{center}
\end{figure}

By fixing knot in combinatorial map, it is convenient to speak about coloring of corners in two colors , correspondingly one color for corners in $C_1$, say, green, and other color for corners in $C_2$, say, red.

\begin{figure}[h]
\begin{center}
\scalebox{.30}{\includegraphics{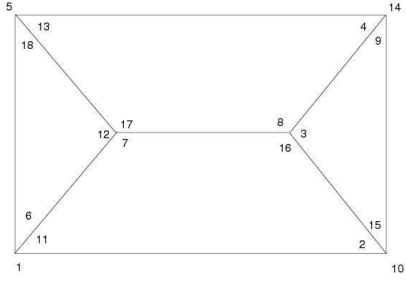}}
\caption{Example of combinatorial map with normalized knot. $P=(1 6 \bar{1})(2 \bar{5} \bar{0})(3 \bar{6} 8)(4 \bar{4} 9)(5 \bar{3} \bar{8})(7 \bar{2} \bar{7})$.  Knot's characteristic $\langle 9 \rangle$, i.e., compact knot. Inner edge involution is partitioned into cut edges $\pi_1=(1 2)(7 8)(\bar{3} \bar{4})$ and cycle edges $\pi_2=(3 4)(5 6)(9 \bar{0})(\bar{1} \bar{2})(\bar{5} \bar{6})(\bar{7} \bar{8})$. Fixed by  knot cycles in map $P$ are $\gamma_1=(1 5 \bar{3} \bar{7} 7 \bar{1})(3 \bar{5} 9)$ and $\gamma_2=(2 \bar{6} 8 4 \bar{4} \bar{0})(6 \bar{2} \bar{8})$. Symmetric knotting $A$ is $(1 5 \bar{3} \bar{7} 7 \bar{1})(2 6 \bar{4} \bar{8} 8 \bar{2})(3 \bar{5} 9)(4 \bar{6} \bar{0})$. Edge involution corresponding to possible value of square root $\nu_1 = (\sqrt{A})_1 = \pi^{-\gamma_1}$   is $(1 6)(2 \bar{1})(3 \bar{6})(4 9)(5 \bar{4})(7 \bar{2})(8 \bar{7})(\bar{0} \bar{5})(\bar{3} \bar{8})$. $\delta (=\pi^{\gamma_1}=\nu_2=(\sqrt{A})_2)$ is $(1 \bar{2})(2 5)(3 \bar{0})(4 \bar{5})(6 \bar{3})(7 \bar{8})(8 \bar{1})(9 \bar{6})(\bar{4} \bar{7})$. Edge structuring knot is $\varepsilon=(1 2 \bar{1} \bar{2} 7 8 \bar{7} \bar{8} \bar{3} \bar{4} 5 6)(3 4 9 \bar{0} \bar{5} \bar{6})$, with knot's characteristic $\langle6,3\rangle$.}
\label{fig2}
\end{center}
\end{figure}

Fixing knot in combinatorial map edges are partitioned into cut edges and cycle edges, correspondingly, $\pi_1$ and $\pi_2$, so that $P \pi_1$, considered as new map, has alternatingly colored orbits and has only cycle edges and $P \pi_2$ similarly considered has mono colored orbits and has only cut edges, and $\pi = \pi_1 \pi_2$. See examples of combinatorial maps with fixed knot in figures \ref{fig1} and \ref{fig2}. Map $P$ may be expressed in form $P=\gamma_1 \gamma_2 \pi_2$, where $\gamma_1$ and $\gamma_2$ are  acting on sets $C_1$ and $C_2$ correspondingly.

Combinatorial map $(P,Q)$ may be expressed as $$P= \mu \times \alpha = \mu \times A \times \pi_1,$$
where $\alpha$ is called {\em knotting} \cite{ze 94}, and $A=\alpha \times \pi_1$ is knotting in symmetric form, i.e., $A=\gamma_1 \gamma_1^\pi$ \cite{ze 09}.

Combinatorial knot may be expressed in form \cite{ze 09} $$\mu=\gamma_2 \pi \gamma_1^{-1}.$$

More useful expressions for combinatorial maps may be found in \cite{ze 09}.

\subsection{Edge structuring knot}

We say that two knots are {\em equivalent} if eventual change of orientation of some orbits lead to equality or them.

We are going to use following theorem.

\begin{thm} Let knot $\varepsilon$ be built from involutions $\pi$ and $\delta = (\pi^{\gamma_1})$: $\varepsilon=\prmb \pi \delta$.
Let knot $\nu$ be one of square roots of $A$, i.e., $\nu^2=A$. Then $\nu$ is equivalent to $\varepsilon$.
\end{thm}
\begin{proof} We use equalities $A=\gamma_1 \gamma_1^\pi$ \cite{ze 09}, and $\varepsilon= \prmb \pi {\pi^{\gamma_1}} = \prmb {\gamma_2 \pi} {\pi \gamma_1}$.
Let us square knot $\varepsilon$:
$$\varepsilon^2 = \prmb {\gamma_2 \pi \pi \gamma_1} {\pi \gamma_1 \gamma_2 \pi} = \prmb {\gamma_2 \gamma_1} {\gamma_1 \gamma_2} \cong A.$$
\end{proof}

We call knot $\varepsilon=\prmb \pi \delta$ {\em edge structuring knot}.

\section{Zigzag walk as an invariant in the graph on orientable surface}

\begin{figure}[h]
\begin{center}
\scalebox{.40}{\includegraphics{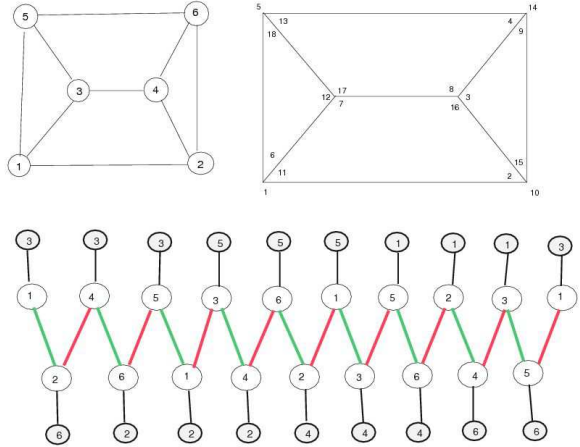}}
\caption{Compact knot example. Zigzag walk in envelope graph $\bar{C}_6$ and corresponding combinatorial map with normalized knot.}
\label{fig3}
\end{center}
\end{figure}

Combinatorial maps has in correspondence graphs on surfaces, i.e., graphs with fixed neighboring edge orders or edge (or adjacent vertex) rotations around vertices. Similarly, other rotational functions of combinatorial maps may have graph theoretical objects, and even several, in correspondence. Combinatorial knot \cite{ze 98,ze 94,ze 08} has in correspondence so called {\em zigzag walk} \cite{bon 95,rore 78} in the graph on orientable  surface.

Zigzag walk in graph may be easily defined. Edges in the walk are to be visited two times and colored with two different colors, say, green and red. Let an edge have not yet received green color: let choose direction on this edge, and start with coloring it green, and choose most left edge in the neighborhood of head, and color it red, and then choose most right edge in neighborhood of head of new edge, and color it green, and so on until walk closes. Let us say that zigzag walk has $k$ components if  we must resume walk $k$ times to cover all edges with two colors. Let zigzag walk as object be  $k$, in number, cyclical sequences of colored edges, or rotation on doubled set of edges with $k$ orbits. We must observe that coloring in zigzag walk edges alternatingly green and red, we choose them most left and most right edge alternatively in the neighborhood of the head of the walk.

\begin{figure}[h]
\begin{center}
\scalebox{.60}{\includegraphics{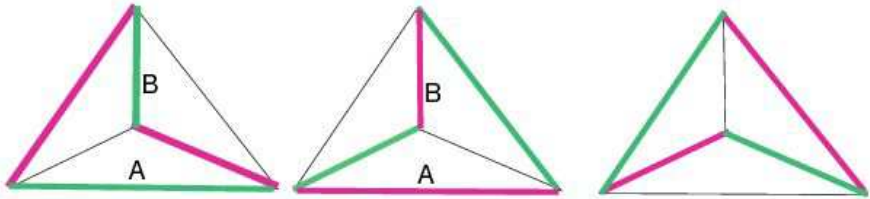}}
\caption{Example of zigzag walk perambulation with three components in $K_4$ graph. We must notice that zigzag walk doesn't give rotation of ordered set of edges: edge $A$ in first and second component perambulated in different directions gives for edge $B$ identic orientations. Combinatorial map theory says that cut edges are visited in different directions, and cycle edges in coincident directions. Otherwise, it is easy to see that each edge is visited just two times. }
\label{fig4}
\end{center}
\end{figure}

\begin{thm} Zigzag walk always ends with all edges in graph being visited just two times with two different colors.
\end{thm}
\begin{proof} Let us observe that zigzag walk in even cycle has two components visiting each edge once, but in odd cycle it has one component visiting each edge twice. Thus, in each component zigzag walk has to visit even number of edges. The same remains true for arbitrary graph, because we may eliminate all edges not visited by the walk leaving only cycle (or closed walk) that zigzag walk has just perambulated. The graph that was perambulated is a special graph that may be perambulated with one cycle not visiting edge more than twice. As in case of simple cycles walk starting with green edge are to end with red edge, [because sparse edges as neighboring are alternatingly on both side of the walk, and walk should be closed in the same way, scilicet, edges colored alternatingly, thus being of even length].
\end{proof}

Zigzag walk is graph theoretical invariant if directions in walk are ignored.

\subsection{Equivalence classes of edges caused by zigzag walk}
Zigzag walk is graph theoretical invariant of the graph in sense that, fixing vertex rotation in the graph, zigzag walk is fixed uniquely up to choosing direction of the walk in each of its components. Thus, if zigzag walk has $k$ components, then actually $2^k$ different zigzag walks of varying directions are possible.

In response to this fact let us say that edges colored in one color during one component of zigzag walk are equivalent. Then zigzag walk causes $2 \times k$ equivalence classes of edges. The colors in each separate component may be interchanged thus giving $2^k$ possible zigzag walks in $k$ components. We notice thus that one equivalence classes' partitioning, corresponding to zigzag walk in general, gives $2^k$ formally different zigzag walks by taking into account order of perambulation of edges in each zigzag walk. To have convenient graph invariant it is more useful to consider these $2^k$ different zigzag walks as one graph invariant. Thus, we call zigzag walk in general all class of zigzag walks that have in correspondence the same equivalence classes of partitioning of doubled edge set of the graph. Particular zigzag walk would be rotation on doubled edge set with $k$ orbits, but zigzag walk in general would be the same rotation with ignoring order of rotations in orbits.

Thus, we come close to consider zigzag walk as combinatorial map expression. But we must notice that we can't replace doubled edge set with the  set of oriented edge  because zigzag walk doesn't guarantee edge's perambulation in both directions for all edges.

Example of zigzag walk of the graph $K_4$ that consists from three components is shown in figure \ref{fig1}. Graph $\overline{C_6}$, envelope graph, has one component zigzag walk, see figure \ref{fig2}. Dual graph $W_4$ should have the same property.

We are encouraged to think in way that vertex rotation causes another rotational relation, scilicet, zigzag walk in the graph. If it would work reversely it would allow to consider zigzag walk rotation as basic, and other rotational relations as secondary. We are going to show that this works.

\subsection{Computation of zigzag walk in graph on surface}

It is very instructive to look on procedure of calculation of zigzag walk in order to reveal its rotational character.

Let $E$ is set of edges of the graph and functions next(a,b) and previous(a,b) give next and previous cyclic neighbors for vertex $a$ with respect to vertex $b$.

\textbf{procedure of zigzag walk:}

\textbf{instr 1:} \textbf{while} edge $(a,b)$ is found that is not yet colored green:

\{\textbf{instr 2:} \textbf{while} (edge $(next(b,a),b)$ is not colored green) \textbf{repeat}:

\{color edge $(a,b)$ green;

color edge $(next(b,a),b)$ red;

\textbf{set} $a:=next(b,a)$;

\textbf{set} $b:=previous(a,b)$

\}\}

We initialize procedure with setting all edges being without any color. Procedure ends when all edges are colored at least twice. We should mark that edges receive both colors.

It is easy to see that this procedure is purely rotational. Scilicet, rotational functions \textit{next} and \textit{previous} cause rotational character of all procedure of perambulation of zigzag walk.

Procedure gives $2^k$ different results if instruction \textbf{instr1} runs $k$ times during computation of zigzag walk.

\section{Combinatorial knot correspondence to zigzag walk}

Combinatorial map theory may only fix rotational facts. In setting correspondence between, say, graph theory and combinatorial map theory, we are able to fix purely rotational type facts/theorems in the first. We show this for case of zigzag walk in the graph with fixed vertex rotations.

\begin{thm} Combinatorial knot of combinatorial map has zigzag walk in the corresponding graph on the orientable surface.
\end{thm}
\begin{proof} Let $\mu=\prmb \pi \rho$ be combinatorial map with inner edge rotation $\pi$ end edge rotation $\rho$. Both are involutions without isolated points. Thus, orbits in $\mu$ are of even length with edges from involutions $\pi$ and $\rho$ alternatingly. Let us have graph $G$ that corresponds to combinatorial map $P$  and the knot $\mu$ is fixed. Then each edge $(a,b)$ of graph have in correspondence transposition $(a_1,b_1)$ from involution $\pi$, and transposition $(a_2,b_2)$ from involution $\rho$. Then must hold that graph vertices $a$ and $b$ have in correspondence vertex rotation that $a_2^P=a_1$, and $b_2^P=b_1$ holds, [or $a_2^P=b_1$, and $b_2^P=a_1$]. Let us assume  first case being right. If we traverse zigzag walk in $P$ taking this edge, corresponding to $(a,b)$, as belonging to $\pi$, we color edge in graph green, and next edge choose from $\rho$, turning as if left, if edge belongs to $\rho$, then we color it red, and turn as if right by choosing next edge from $\pi$. It is easily seen that we repeat in combinatorial map just the same procedure that in case of traveling graph by zigzag walk. The fact that this process ends as expected, in case of permutations is trivial.
\end{proof}

Further we see that combinatorial knot is fixed only up to orientation of orbits: changing orientation of orbits in $\mu$ give as if different knot that in graph would belong to the same equivalence class of zigzag walks, or the same walk in general.

Let us accept convention already used in the proof: edges that correspond to action of $\pi$ we color green, but edges of $\rho$ -- correspondingly we color red.

\subsection{Zigzag walk as green-red- or left-right- or $\pi$-$\rho$-walk}

Establishing correspondence between combinatorial knot in combinatorial map and zigzag walk in graph on surface we establish actually that the same procedure in the graph on surface may be considered as alternation both of direction or color of edge or belonging to involutions $\pi$ and $\rho$. Thus, we get the same zigzag walk process  as left-right-walk, if choosing direction of walk, as red-green-coloring, when choosing color for current edge in walk, and as $\pi$-$\rho$-walk in combinatorial map that as map  corresponds to the graph.

It is convenient to consider combinatorial map definitions as basic where graph theoretical relations are interpretations of the first. This may serve as methodological basis for applications of similar nature.

\section{Normalization of combinatorial knot}

We have come to main point of this work. In work \cite{ze 94} we introduced normalized combinatorial maps assuming for involution $\pi$ some fixed form: we chose increasing numeration, i.e., $\pi = (1 2)(3 4)(5 6)... (2m-1 2m)$ for $m$ edges in geometrical combinatorial map. It was done by purely practical reason: to fix edge set for a graph we don't need double numbering system as it appears to be necessary in combinatorial map approach with two edge involutions, scilicet, $\pi$ and $\rho$. Choosing one of them fixed, say, $\pi$, we may vary with other. Fixing involution $\pi$ leads to one more consequence. If in case of arbitrary involution $\pi$ we need two permutations, say, pair $(P,Q)$,  to characterize combinatorial map, then in case of fixed involution $\pi$ we need only one permutation, scilicet, $P$, because $Q$ may be computed, i.e., $Q = P \times \pi$.
The one more consequence is even more crucial. With fixing involution $\pi$ we establish one-one map between even permutations and combinatorial maps.

Let us fix this fact as theorem.

\begin{thm}\label{th5} Between combinatorial maps and even permutations may be established one-one map.
\end{thm}
\begin{proof} Let us fix involution $\pi$, for all class of combinatorial maps being the same. Combinatorial map $(P,Q)$ may be characterized with one permutation, say, $P$, where second may be calculated, i.e., $Q = P \times \pi$. Thus, every even permutation $p$ may be interpreted as combinatorial map $(p, p \times \pi)$. Reversely, every combinatorial map with fixed inner involution $\pi$ is characterized by single even permutation. Proof is done.
\end{proof}

In this work we established that normalization of combinatorial map may be continued. We may choose augmenting numeration not only  for involution but for all combinatorial knot. By this we may use fact that combinatorial knot has graph theoretical invariant in correspondence, i.e., equivalence class of zigzag walks as one generalized zigzag walk.

\begin{defn} Let us say that combinatorial map has normalized knot if its knot may be written in form of augmenting corner numbers from $1$ to $2m$.
\end{defn}

It is easy to see that combinatorial map with normalized knot is normalized in ordinary sense too. Moreover, edges from involution $\rho$ become by normalization with augmented values too, namely, with every transposition in form $(2i,2i+1)$ , where the addition is performed in some cyclical notion, scilicet, becoming cycled in within orbit. Let see theorem \ref{th6} further. Only, now we loose the property that maps have even permutations in correspondence. Now only some subset of even permutations are maps with normalized knot.

It is easy to see that every map may be renumbered in order to become map with normalized knot.

We may need to use less strong notion of knot normalization. For that reason we define {\em partially normalized knot}.

\begin{defn} Normalized map has partially normalized knot if in knot $\mu$  pairs from involution $\pi$  have the same orientation, increasing or decreasing, for all orbits.
\end{defn}

Partially normalized knot have one orientation edges only from one edge involution whereas in fully normalized knot we require the same orientation for both edge involutions.

See example of map with normalized knot in figure \ref{fig2}.

\begin{figure}[h]
\begin{center}
\scalebox{.60}{\includegraphics{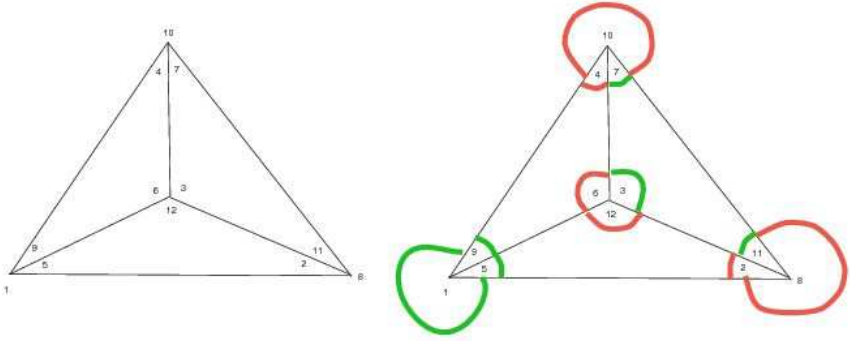}}
\caption{Example of combinatorial map with normalized knot. Knot $\mu$ is equal to $(1 2 3 4)(5 6 7 8)(9 \bar{0} \bar{1} \bar{2})$. Set of corners $C$ is partitioned into subsets of green corners $C_1 = \{1,3,5,7,9,\bar{1}\}$ and red corners $C_2 =  \{2,4,6,8,\bar{0},\bar{2}\}$.
$P = (1 9 5)(2 \bar{1} 8)(3 \bar{2} 6)(4 \bar{0} 7)$, $\alpha = (1 \bar{0} 5 2 9 6)(3 \bar{1} 7)(4 \bar{2} 8)$.
Three distinguished cycles  by fixing the knot in map are easy to be seen: odd numbered green cycles $(1 9 5)$ and $(3 7 11)$ and even numbered red cycle $(2 12 4 10 8)$. }
\label{fig5}
\end{center}
\end{figure}

\section{Main theorems}

Next theorem calculates values for edge rotation $\rho$ in case of normalized knot.

\begin{thm}\label{th6} Let in the process of perambulation of knot edge is visited as $k$-th in the orbit and orbit starts with $l_1$-th edge and ends with $l_2$-th edge.   By knot normalization edge  becomes equal to $$(k+l_1-1, mod(k-1,l_2-l_1)+l_1).$$
\end{thm}
\begin{proof} Orbit starts with edges $(l_1,l_1+1)(l_1+1,l_1+2)$ and ends with edges $(l_2-1,l_2)(l_2,l_1)$. In assertion of theorem arbitrary  edge involution element in orbit is calculated taking into account that $k$-th  edge  in orbit starts with corner $k+l_1-1$, and next corner either is augmented by one or becomes equal to $l_1$.
\end{proof}

Let us consider features of combinatorial maps with partially normalized or fully normalized knot.

Let us consider a theorem in a weaker and stronger form.

\begin{thm}  By knot normalization corner set $C$ is partitioned into even and odd corner sets, i.e., $C_1$ contains odd corners and $C_2$ contains even corners by taking form of knot $\mu=\prmb \pi \rho$.
\end{thm}
\begin{proof} Corners in normalized knot $\mu$ are numbered from $1$ to $2m$ with every next augmented by one. Each odd edge in sequence belongs to inner edge involution $\pi$ and each even edge belongs to edge involution $\rho$. If first corner belongs to $C_1$ and second to $C_2$  then alternatingly theorem's assertion follows.
\end{proof}

\begin{thm}  The set $C$ is partitioned into even and odd corner sets if and only if knot is partially normalized.
\end{thm}
\begin{proof} Let us assume that edges from $\pi$ are in knot oriented in one direction. Let be augmented, say. Then it starts with odd value and in zigzag walk tail of edge is added to subset $C_1$ and head to $C_2$. Then $C_1$ should receive only odd values, and $C_2$ -- only even values.

Let orientation of some edge from $\pi$ is directed otherwise: then its even tail are going to $C_1$ and we have come to contradiction.
\end{proof}

\begin{thm} Let knot be partially normalized. If for  corner $a$ both $a$ and $a^P$ are odd, or even, then edge $(a,a^{\rho})$ is cut edge, otherwise it is cycle edge.
\end{thm}
\begin{proof} If both $a$ and $a^P$ are odd or even they belong to $C_1$, otherwise one to $C_1$ and other to $C_2$.  In first case edge is cut edge, but in second case cycle edge.
\end{proof}

If in assertion of this theorem in place of vertex rotation $P$ we put face rotation $Q$ then \textit{cut edge} and \textit{cycle edge} in assertion should be interchanged.

\subsection{Renumeration of combinatorial map}

Let us have arbitrary combinatorial map $P$ and ask, how to renumerate it in order to have knot normalized in it. Let map $P$ have knot $\mu$, and new map with renumerated corners should be $P'$ with knot $\mu'$.

Let us denote with variable $I$ element from class of permutations  that has orbits of form $(k, k+1, ..., k+l-1)$ with length $l$, i.e., with elements augmented by one. Such permutation may be characterized by lengths of orbits. Thus, some constant permutation $I$ with $n$ orbits is determined by sequence of numbers $<l_1,...,l_n>$ where $l_i$ is length of $i$-th orbit and : $$I=(1 2 ...l_1)({l_1+1} {l_1+2}...{l_1+l_2})({l_1+l_2+1} ...) ..(\sum \limits_{i=1}^{n-1}{l_i}+1...) .$$

It is evident that normalized knot $\mu'$ should be equal with some permutation $I$ that has corresponding even lengths of orbits.

Further, let us find permutation $T$ that should serve as renumerator of corners for $P$ that should become normalized with respect to its knot. We must get $P^T=P'$. Also for knot the same permutation should work in the same way, scilicet, $\mu^T =\mu'=I$. It means that equation $\mu \times T = T \times I$ should hold.

\subsubsection{Calculation of transformation $T$}\label{sec511}

Let us define bijection $B$ that should correspond to knot $\mu$ that is built in the following way: arrange $\mu$ in the way that each orbit starts with edge from involution $\pi$ and concatenate orbits as substrings in a string and assign this string to bijection $B$ so that $i$-th element of $B$ maps to $i$-th element of the string. Then following theorem is true.

\begin{thm} Bijection $B$ considered as substitution is equal to reverse permutation $T^{-1}$.
\end{thm}
\begin{proof} Let $B$ be considered as substitution denoted by the same letter. Then it is easy to see that $I^B$ should be equal to $\mu$ if $I$ is congruent with $\mu$, i.e., $\mu' = I$. But then $T$ should be equal to $B^{-1}$.
\end{proof}

\begin{thm}\label{th12} If $\mu$ has $k$ orbits there are $2^k k!$ ways to built $B$ from knot $\mu$ thus giving way to $2^k k!$ possible transformations $T$ for combinatorial map with this knot.
\end{thm}
\begin{proof} Each orbit from $\mu$ may be glued in string in two ways with two possible orientations that gives $2^k$ ways for $k$ orbits. But oriented orbits may be combined into string in $k!$ ways. Doing all this independently gives altogether  $2^k k!$ possible transformations $T$ with  $2^k k!$  possible ways to transform map into knot normalized map.
\end{proof}

\begin{thm} Permutation $T$ belongs to class of maps from set $\Pi$ \cite{ze 96}, i.e., $T^\pi=T$.
\end{thm}
\begin{proof} For substitution $B$ pair of successive elements $(B[2i-1],B[2i])$ always belongs either to inner edge involution $\pi$, or edge involution $\rho$. The same holds for reverse of $B$, that is equal to $T$. This gives equality $T^\pi=T$, and $T$ as map belongs to $\Pi$, i.e., set of self congruent maps.
\end{proof}

\subsection{Class of combinatorial maps with normalized knot}

Let us recall that combinatorial map $P$ may be expressed as $P = \mu \times \alpha = \mu \times A \times \pi_1$, where small $\alpha$ and big $A$ are so called knottings. In counting corresponding spaces, let us recall equality $(2m)!=(2m-1)!! \times 2^m \times m!$: $m!$ stands for size of space for $A$; $(2m-1)!! \times 2^m$ for knot space where $ 2^m$ stands for $\pi$ subsets, e.g., $\pi_1$, and $(2m-1)!!$ for variation of knot in $m!$ space with both changing orientation of edges.

Normalized combinatorial maps are $(2m-1)!$ in number. Maps with normalized knots we loose this simple expression. For this count we use formula $P = \mu \times A \times \pi_1$ assuming that $\mu$ is normalized. For normalized knots  may be counted ($\textbf{N}(\mu)$) using partition function \cite{part}:
$$pr(n,j)=1+\sum\limits_{k=2}^j{\sum\limits_i^{n \div k} {pr(n-i*k,k-1)}}$$ with $$pr(1,1)=1,$$
and partition number then is
$$p(n)=pr(n,n),$$
i.e., $$\textbf{N}(\mu)=p(m)\space \ k! \ 2^k,$$
using theorem \ref{th12}.

\begin{figure}[h]
\begin{center}
\scalebox{.40}{\includegraphics{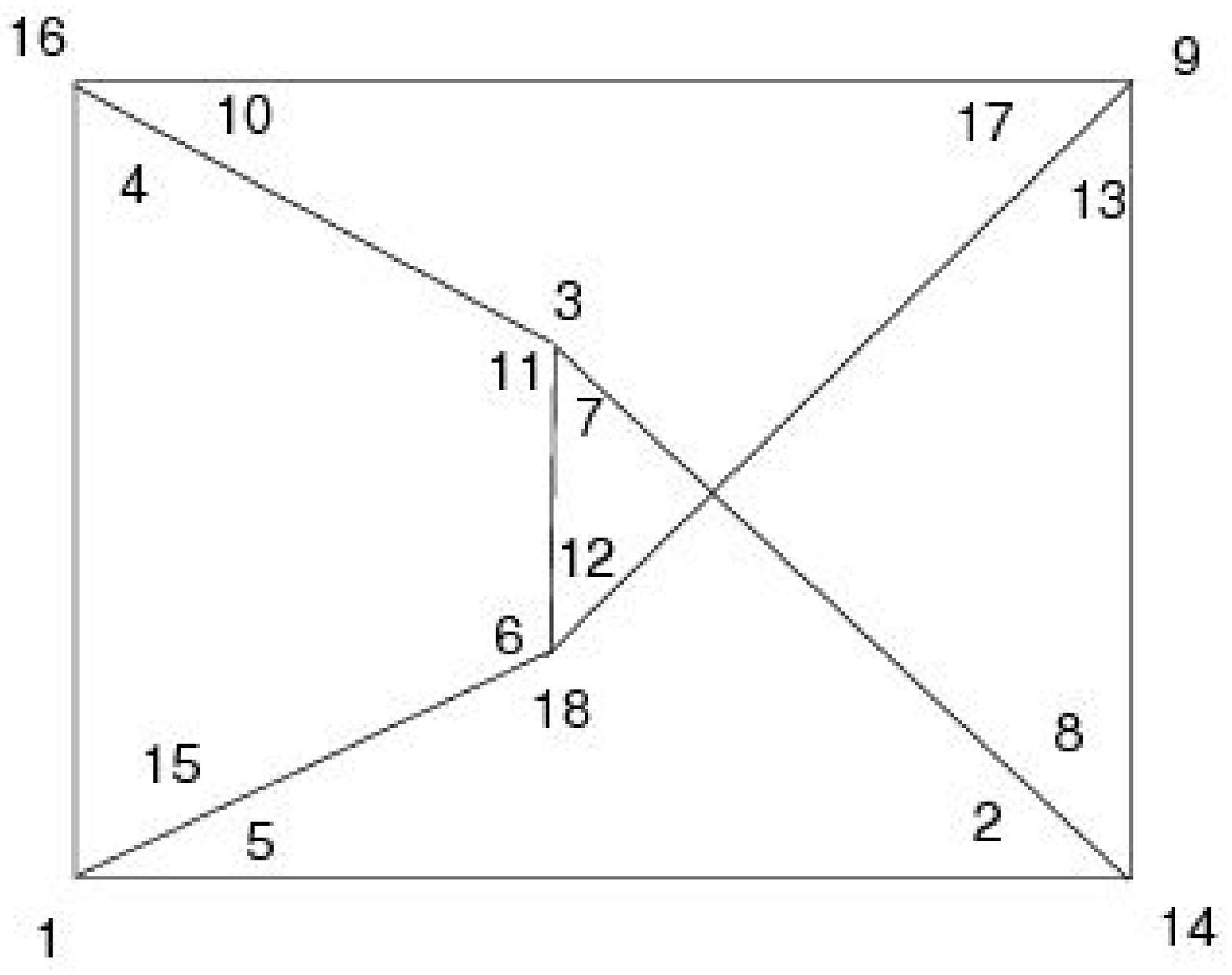}}
\caption{Example of combinatorial map with normalized knot for graph $K_{3,3}$. Knot's characteristic is $\langle2,5,2\rangle$. Edge structuring knot $\varepsilon$ is (1 2 5 6 15 16)(3 4 11 12 7 8)(9 10 17 18 13 14) of knot's characteristic $\langle3,3,3\rangle$. Symmetric knotting is (1 15 5)(2 16 6)(3 7 11)(4 8 12)(9 13 17)(10 14 18). Cycle edge set is empty, i.e., $\pi_1=\pi$.
}
\label{fig6}
\end{center}
\end{figure}

\begin{figure}[h]
\begin{center}
\scalebox{.40}{\includegraphics{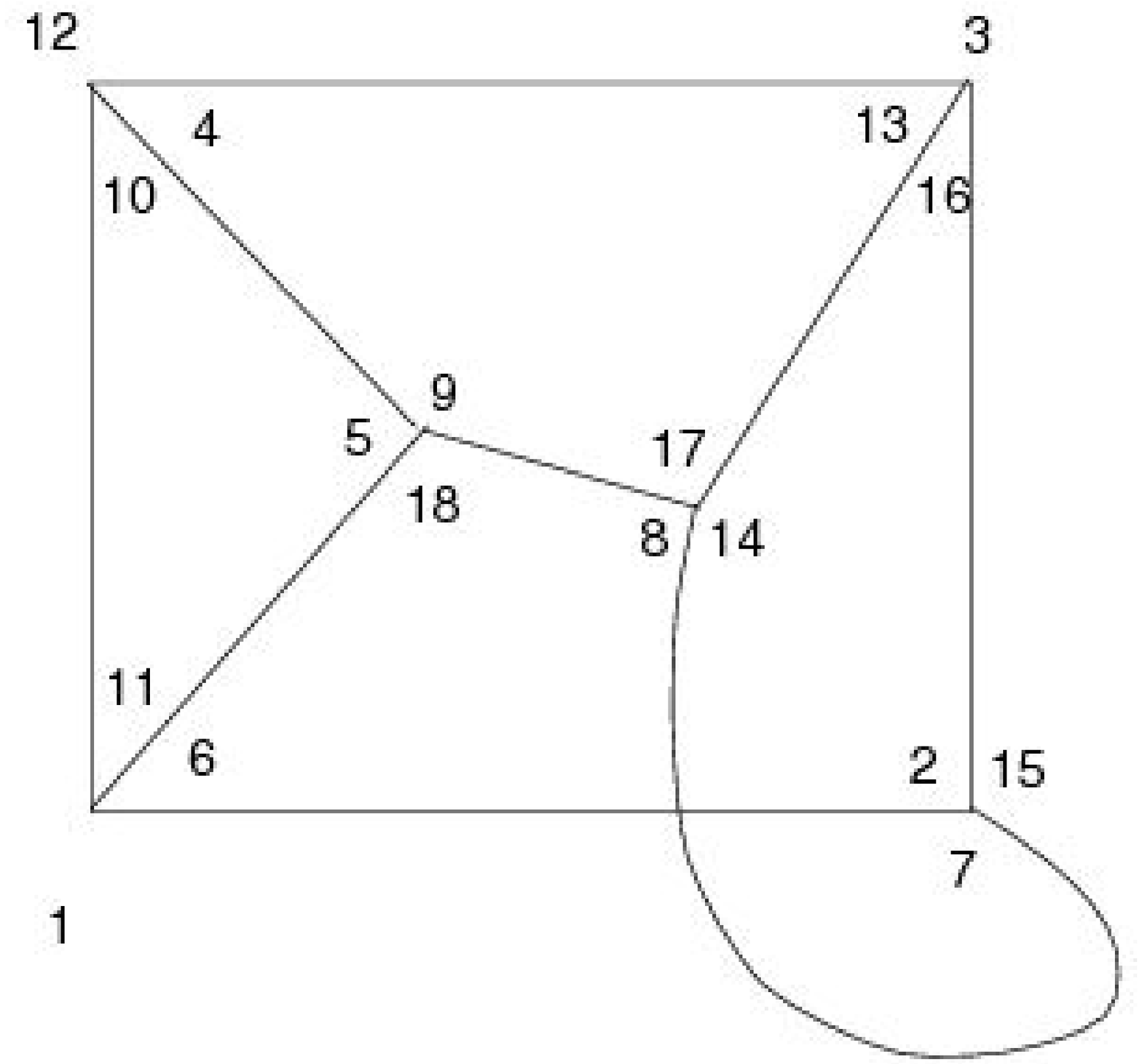}}
\caption{Example of combinatorial map with normalized knot. Knot's characteristic is $\langle5,4\rangle$. Edge structuring knot  $\varepsilon$ is (1 2 7 8 15 16 3 4 13 14 17 18 9 10 5 6 11 12). Symmetric knotting is (1 11 5 9 17 13 3 15 7)(2 12 6 10 18 14 4 16 8). Cut edges are these of $\pi_1$: (3 4)(7 8)(9 10)(11 12).}
\label{fig7}
\end{center}
\end{figure}

In work \cite{ze 08} we have showed that combinatorial map, say, $P$, may be expressed as multiplication of knots, i.e., $P=\mu \ \varepsilon^2 \ \pi_1$. Graph on surface may be characterized by multiplication of three knots. Knot $\mu$ is responsible for zigzag walk, knot $\varepsilon$ is caused by cycles arising from knot $\mu$ fixation, and $pi_1$ is cut edge involution. Thus, zigzag fixation causes fixed cycles and cut (and cycle) edges: this in combinatorial map theory setting gives that $\mu$ rotation causes two rotations, i.e., these of $\varepsilon$ and $pi_1$.

Besides, we may apply normalization of knot not only to $\mu$ but to $\varepsilon$ too.  Of course these two operations are not compatible, i.e., we may normalize either one or other, not both. Nevertheless, theorem is true.

\begin{thm} If $\mu$ or $\varepsilon$ is (fully) normalized other is partially normalized.
\end{thm}
\begin{proof} This  follows from fact that partitioning $C=C_1 \cup C_2$ is common for both operations of normalization.
\end{proof}

\section{Conclusions}

Graph on surface may be characterized by multiplication of three knots \cite{ze 09}. Knot $\mu$ is responsible for zigzag walk, knot $\varepsilon$ is caused by cycles arising from knot $\mu$ fixation, and $pi_1$ is cut edge involution. Thus, zigzag fixation causes fixed cycles and cut (and cycle) edges: this in combinatorial map theory setting gives that $\mu$ rotation causes two rotations, i.e., these of $\varepsilon$ and $\pi_1$.

Combinatorial map theory may be developed in way giving concise correspondence to graphs topological picture. In the same time our aim was to demonstrate that rotational relations should be considered as more fundamental than traditional set theoretical settings in conventional graph theory.

Combinatorial maps may be used in building computer programs for graph theoretical applications. Data structures may be built using rotational functions of combinatorial maps. We have used these aspects to build algorithms used in calculations mentioned in works  \cite{ze1 04,ze 09}.

\section{Acknowledgements}

I would like to thank Paulis \c{K}ikusts for discussions of combinatorial map theory aspects at regular seminaries of graph theory at Institute of Mathematics and Computer science.

\end{document}